\documentclass[a4paper,10pt,conference]{ieeeconf}
\usepackage{graphicx}
\usepackage{amssymb}
\usepackage{amsmath}
\usepackage{epstopdf}
\IEEEoverridecommandlockouts

\newtheorem{definition}{Definition}
\newtheorem{theorem}{Theorem}

\title{Optimal Sensor and Actuator Placement in \\Complex Dynamical Networks}
\author{Tyler H. Summers and John Lygeros
\thanks{T.H. Summers and J. Lygeros are with the Automatic Control Laboratory, ETH Zurich, Switzerland, email: \{tsummers,jlygeros\}@control.ee.ethz.ch This work is partially supported by the ETH Zurich Postdoctoral Fellowship Program. }}
\date{\today}                                           % Activate to display a given date or no date

\begin{document}
\maketitle

\begin{abstract}
Controllability and observability have long been recognized as fundamental structural properties of dynamical systems, but have recently seen renewed interest in the context of large, complex networks of dynamical systems. A basic problem is sensor and actuator placement: choose a subset from a finite set of possible placements to optimize some real-valued controllability and observability metrics of the network. Surprisingly little is known about the structure of such combinatorial optimization problems. In this paper, we show that an important class of metrics based on the controllability and observability Gramians has a strong structural property that allows efficient global optimization: the mapping from possible placements to the trace of the associated Gramian is a \emph{modular set function}. We illustrate the results via placement of power electronic actuators in a model of the European power grid.
\end{abstract}

\section{Introduction}
% the setting: large complex networks
%One of the grand challenges facing modern society is the understanding and efficient control of large, complex technological networks, in which many dynamical subsystems interact, including, for example, power grids, the Internet, and transportation networks. Moreover, there has been a recent convergence of technological networks with social networks, resulting in an explosion of data and new efforts in mathematical modeling aimed at quantitative understanding of complex social network processes. This has led to the formulation of decision and control problems in social networks, such as maximizing or minimizing the spread of ``contagion'' (maximizing in the context of product adoption, healthy behaviors, etc.; minimizing in the context of disease, unhealthy behaviors, financial defaults, etc.).

% literature review: 
% controllability in networks
Controllability and observability have been recognized as fundamental structural properties of dynamical systems since the seminal work of Kalman in 1960 \cite{kalman1960contributions}, but have recently seen renewed interest in the context of large, complex networks, such as power grids, the Internet, transportation networks, and social networks. A prominent example of this recent interest is \cite{liu2011controllability}, which, based on Kalman's rank condition and the idea of structural controllability, presents a graph theoretic maximum matching method to efficiently identify a minimal set of so-called driver nodes through which time-varying control inputs can move the system around the entire state space (i.e., render the system controllable). The method of \cite{liu2011controllability} is applied across a range of technological and social systems, leading to several interesting  and surprising conclusions. Using a metric of controllability given by the fraction of driver nodes in the minimal set required for complete controllability, it is shown that sparse inhomogeneous networks are difficult to control while dense homogeneous networks are easier. It is also shown that the minimum number of driver nodes is determined mainly by the degree distribution of the network. Many other studies of controllability in complex networks have followed, including \cite{rajapakse2011dynamics,nepusz2012controlling,wang2012optimizing,tang2012epc,tang2012identifying}. 

One issue with the approach taken by \cite{liu2011controllability} and much of the follow up work is that the quantitative notion of controllability discussed in \cite{liu2011controllability} (namely, the number/fraction of required driver nodes) is rather crude in some settings. This was noted, for example, by \cite{muller2011few} in response to the surprising result in \cite{liu2011controllability} that genetic regulatory networks seem to require many driver nodes, which apparently contradicts other findings in biological literature on cellular reprogramming. % sensor and actuator placement/selection
Rather than finding a set of driver nodes that would render a network completely controllable, a more appropriate strategy might be to choose from a finite set of possible actuators and sensors placements to optimize some real-valued controllability and observability metrics of the network, which is not considered in \cite{liu2011controllability}. There is a variety of more sophisticated metrics for controllability and observability that have been proposed in the systems and control literature on sensor and actuator placement or selection problems in dynamical systems; see, e.g., the survey paper \cite{van2001review}. One important class of metrics involves the controllability and observability \emph{Gramians}, which are symmetric positive semidefinite matrices whose structure defines energy-related notions of controllability and observability.

While a variety of metrics have been proposed \cite{van2001review}, the corresponding combinatorial optimization problems for sensor and actuator placement are less well-understood. These can be solved by brute force for small problems by testing all possible placement combinations. However, for problems that arise in large networks, testing all combinations quickly becomes infeasible and so the problems are generally thought to be very difficult combinatorial optimization problems, requiring inefficient integer programming techniques in general. Some techniques based on efficient $\ell_1$ heuristics have been proposed, e.g. \cite{hassibi1999low}, but it is not clear when they work.

% contributions
In the present paper, we show that one important class of metrics of controllability and observability, previously thought to lead to difficult combinatorial optimization problems \cite{van2001review}, can be in fact easily optimized, even for very large networks. In particular, we show that the mapping from subsets of possible actuator/sensor placements to any linear function of the associated controllability or observability Gramian has a strong structural property: it is a \emph{modular set function}. We also describe how this observation defines a new dynamic network centrality measure for networks whose dynamics are described by linear models, assigning a control energy-related ``importance'' value to each node in the network. We illustrate the results in power electronic actuator placement in a model of the European power grid.

%%% sub modularity - to be considered later
%These notes explore placement of sensors and actuators in complex networks from the perspective of systems and control theory. A key question that appears not to have been answered is which metrics are modular or submodular. Modular functions are easy to optimize; submodular functions are hard to optimize, but come with suboptimality guarantees. We show that any linear function of the controllability or observability Gramian, which is related to the $H_2$ system norm, is modular and therefore easily optimized.

% structure
The rest of the paper is organized as follows. Section II reviews basics of linear dynamical systems and controllability. Section III introduces the notions of modular, submodular, and supermodular set functions and shows that the set function mapping possible actuator placements to any linear function of the controllability Gramian is modular. Section IV presents a case study in a power network. Section V gives concluding remarks and and outlook for future research.

%%% Linear Dynamical Systems %%%
\section{Linear Models of Network Dynamics}
This section defines a linear model for network dynamics and reviews and interprets metrics for controllability based on the controllability Gramian. The material in this section is mostly standard and can be found in many texts on linear system theory, e.g. \cite{kailath1980linear,callier1991linear}; we discuss the material mostly to set our notation. Since controllability and observability are dual properties \cite{kalman1959general}, we focus only on controllability; all of the results have analogous counterparts and interpretations for observability. 

While virtually all real systems have nonlinear dynamics, there is a local structural equivalence between nonlinear models and associated linearized models (via the Hartman-Grobman theorem), and the resulting linear models are widely used across many engineering and science disciplines. We therefore focus on linear, time-invariant dynamical network models, in which the dynamics are given by
\begin{equation} \label{linearmodel}
\begin{aligned}
\dot{x}(t) &= A x(t) + B u(t), \quad x(0) = x_0, \\
         y(t) &= C x(t)
\end{aligned}
\end{equation}
where $x(t) \in \mathbf{R}^n$ represents the state of the network and $u(t) \in \mathbf{R}^m $ represents the control inputs that can be used to influence the networks dynamics. For example, $x(t)$ might represent voltages, currents, or frequencies in devices in a power grid, species concentrations in a genetic regulatory network, or individual opinions or propensities for product adoption in a social network. 
%The matrix $A \in \mathbf{R}^{n\times n}$ is called the dynamics matrix and describes how the nodes in the network influence one another dynamically. The matrix $B \in \mathbf{R}^{n\times m}$ is called the input matrix and describes how the external inputs influence the states of certain nodes in the network. (Note that a single input might influence a single node or multiple nodes; this is described by the structure of the $B$ matrix columns). 
The dynamics matrix induces a graph on the nodes of the network:  there is an edge between two nodes $i$ and $j$ if $a_{ij}$ is non-zero. The matrix $C \in \mathbf{R}^{p\times n}$ is typically interpreted as a set of linear state measurements, but here we will interpret it as a weight matrix whose rows define important directions in the state space.

%The general solution to the linear dynamical system is given by
%\begin{equation}
%x(t) = e^{At} x_0 + \int_0^t e^{A(t-\tau)} B u(\tau) d\tau,
%\end{equation}
%where $e^{At} = \sum_{k=1}^\infty \frac{1}{k!} (At)^k $ is the matrix exponential or state transition matrix and can be computed in a variety of ways.

\subsection{Controllability}
\begin{definition}[Controllability]
A dynamical system is called \emph{controllable} over a time interval $[0,t]$ if given any states $x_0$, $x_1 \in \mathbf{R}^n$, there exists an input $u(\cdot):[0,t] \rightarrow \mathbf{R}^m$ that drives the system from $x_0$ at time $0$ to $x_1$ at time $t$. 
\end{definition}
%For linear systems, controllability is equivalent to a more specific question of reachability that involves the set of system states that can be reached from the origin in finite time with a time-varying control input (i.e., we can take $x_0=0$ in Definition 1 without loss of generality). The set of states reachable from the origin in time $t$ is given by
%\begin{equation}
%\mathcal{R}_t  = \left\{ \int_0^t e^{A(t-\tau)} B u(\tau) d\tau \ \vert \ u(\cdot):[0,t] \rightarrow \mathbf{R}^m \right\} \subseteq \mathbf{R}^n.
%\end{equation}
%We have the following basic result and immediate corollary, which is sometimes called Kalman's rank condition \cite{kalman1960contributions}:
%\begin{theorem} \label{ctrb1}
%For any $t > 0$, we have $\mathcal{R}_t = \bigcup_{t\geq 0} \mathcal{R}_t = \text{range}[B, AB, ..., A^{n-1} B]$.
%\end{theorem}
%\begin{corollary}
%A linear dynamical system is controllable iff $[B, AB, ..., A^{n-1} B]$ is full rank.
%\end{corollary}
Kalman's well-known rank condition states that a linear dynamical system is controllable if and only if $[B, AB, ..., A^{n-1} B]$ is full rank. Since rank is a \emph{generic} property of a matrix, it has the same value for almost all values of the non-zero entries of $A$ and $B$ (assuming that the non-zero entries are independent). The controllability property is thus at its core a structural property of the graph defined by $A$ and $B$, as captured in the graph-theoretic concept of \emph{structural controllability} described by Lin in \cite{lin1974structural}, which underpins the recent results of \cite{liu2011controllability}. However, while Kalman's rank condition is widely used, it only gives a \emph{binary} metric for controllability. It is interesting to consider more sophisticated quantitative metrics for controllability in complex networks. 

\subsection{An energy-related controllability metric}
Every actuator in a real system is energy limited, so an important class of metrics of controllability deals with the amount of input energy required to reach a given state from the origin. In particular, we can pose the following optimal control problem seeking the minimum energy input that drives the system from the origin to a final state $x_f$ at time $t$:
\begin{equation} \label{optprob}
\begin{aligned}
 \underset{{u(\cdot) \in \mathcal{L}_2  }}{\text{minimize}} & &&\int_0^t \Vert u(\tau) \Vert^2 d\tau \\
 \text{subject to} & && \dot{x}(t) = Ax(t) + Bu(t), \\
 			  & && x(0) = 0, \quad x(t) = x_f
 \end{aligned}
\end{equation}
Standard methods from optimal control theory can be used to derive the solution. % e.g., the maximum principle or the Hamilton-Jacobi-Bellman equation. 
If the system is controllable, the optimal input has the form 
\begin{equation}
\begin{aligned}
u^*(\tau) = B^T e^{A^T(t-\tau)} &\left( \int_0^t e^{A \sigma} B B^T e^{A^T \sigma} d\sigma \right)^{-1} x_f, \\
 &0 \leq \tau \leq t
\end{aligned}
\end{equation}
and the resulting minimum energy is
\begin{equation}
\int_0^t \Vert u^*(\tau) \Vert^2 d\tau = x_f^T \left(  \int_0^t e^{A \sigma} B B^T e^{A^T \sigma} d\sigma \right)^{-1} x_f.
\end{equation}
The matrix
\begin{equation}
W_c(t) = \int_0^t e^{A \tau} B B^T e^{A^T \tau} d\tau \ \in \mathbf{R}^{n\times n}
\end{equation}
is called the \emph{controllability Gramian} at time $t$. The controllability Gramian is positive semidefinite and has the same rank as $[B, AB, ..., A^{n-1} B]$.  It defines an ellipsoid in the state space
\begin{equation}
\mathcal{E}_{min}(t) = \left\{ x \in \mathbf{R}^n \vert  x^T W_c(t)^{-1} x \leq 1 \right\}
\end{equation}
that contains the set of states reachable in $t$ seconds with one unit or less of input energy. The eigenvectors and corresponding eigenvalues of $W_c$ define the semi-axes and corresponding semi-axis lengths of the ellipsoid. Eigenvectors of $W_c$ associated small eigenvalues (large eigenvalues of $W_c^{-1}$) define directions in the state space that are less controllable (require large input energy to reach), and eigenvectors of $W_c$ associated with large eigenvalues (small eigenvalues of $W_c^{-1}$) define directions in the state space that are more controllable (require small input energy to reach). 

For stable systems, the state transition matrix $e^{At}$ comprises decaying exponentials, so a finite positive definite limit of the controllability Gramian always exists and is given by
\begin{equation}
W_c = \int_0^\infty e^{A \tau} B B^T e^{A^T \tau} d\tau \ \in \mathbf{R}^{n\times n}
\end{equation}
This matrix defines an ellipsoid in the state space that gives the states reachable with one unit or less of energy, regardless of time. This infinite-horizon controllability Gramian can be computed by solving a Lyapunov equation
\begin{equation}
A W_c + W_c A^T + B B^T = 0,
\end{equation}
which is a system of linear equations and is therefore easily solvable, even for large systems. Specialized algorithms have been developed to compute the solution; see, e.g., \cite{bartels1972solution,hammarling1982numerical}.

The controllability Gramian gives a more sophisticated energy-related quantitative picture of controllability, but we still need to form a scalar metric for $W_c$, which is a positive semidefinite matrix. We want $W_c$ ``large'' so that $W_c^{-1}$ is ``small'', requiring small amount of input energy to move around the state space. There are a number of possible metrics for the size of $W_c$, including minimum eigenvalue, determinant, trace, sums/products of the first $k$ eigenvalues, etc.. We focus here on the trace metric, which as we show below, has interesting interpretations in terms of average energy and linear system norms and has a strong structural property for actuator placement problems.

\subsection{Interpretations of trace$(W_c)$}
\textbf{Average Energy}: The average value of the minimum control energy over the unit hypersphere is proportional to the trace of $W_c^{-1}$:
\begin{equation}
\frac{\int_{\Vert x \Vert = 1} x^T W_c^{-1} x dx}{\int_{\Vert x \Vert = 1} dx} = \frac{1}{n} \textbf{tr} W_c^{-1}.
\end{equation}
The trace of $W_c$ is inversely related to the trace of $W_c^{-1}$, so maximizing $\textbf{tr} (W_c)$ effectively minimizes the average energy required to move around the state space in all directions. Maximizing a weighted trace, representing any linear function of a matrix, minimizes a weighted average energy required to move around the state space, with certain directions weighted differently, which can be encoded into the $C$ matrix.

\textbf{System $H_2$ norm}:  An important norm of a (stable) linear dynamical system is defined as
\begin{equation}
\Vert H \Vert_2  = \left( \textbf{tr} \frac{1}{2\pi} \int_{-\infty}^\infty H(j\omega) H(j\omega)^* d\omega  \right)^{1/2}.
\end{equation}
This can be interpreted as the RMS response of the system when it is driven by a white noise input. By the Parseval theorem, this is also given by
\begin{equation}
\Vert H \Vert_2 = \left( \textbf{tr} \int_0^\infty h(t)^T h(t) dt \right)^{1/2} = \Vert h(t) \Vert_2,
\end{equation}
where $h(t) = C e^{At} B$ is the impulse response matrix. Thus, the $H_2$ norm can also be interpreted as the $L_2$ norm, or energy, of the system response to a unit impulse input.

The connection between the system $H_2$ norm and the controllability Gramian can be seen as follows:
\begin{equation}
\begin{aligned}
\Vert H \Vert_2^2  &= \textbf{tr} \left( \int_0^\infty h(t) h(t)^T dt \right) \\
			     &= \textbf{tr} \left( C \int_0^\infty e^{A t} B B^T e^{A^T t} dt C^T \right) \\
			     &= \textbf{tr} (C W_c C^T)
\end{aligned}
\end{equation}
i.e., the system $H_2$ norm is a weighted trace of the controllability Gramian.

%For linear dynamical systems, a complete and elegant theory emerged from the work of Kalman and others in the 1960s (refs) and has been a standard concept in the field ever since. 
%Sensor and actuator placement problems form a surprisingly small subset of the the control theory literature. Much of the literature focuses on the control design, assuming that the placements of sensors and actuators are given. 
%
%We would like to use this more sophisticated quantitative notion to guide actuator placement in complex networks. The next section discusses the structure of these optimization problems.

\section{Optimal Sensor and Actuator Placement in Networks}

\subsection{Set Functions, Modularity, and Submodularity}
%Introduction to submodularity. Submodularity as the combinatorial analog to convexity in continuous optimization.

Sensor and actuator placement problems can be formulated as \emph{set function} optimization problems. For a given finite set $V = \{1,...,M \}$, a \emph{set function} $f: 2^V \rightarrow \mathbf{R}$ assigns a real number to each subset of $V$. In our setting, the elements of $V$ represent potential locations for the placement of sensors or actuators in a dynamical system, and the function $f$ is a metric for how controllable or observable the system is for a given set of placements, which is to be maximized.

We consider set function optimization problems of the form
\begin{equation} \label{optprob}
 \underset{{S \subseteq V, \ |S| = k }}{\text{maximize}} \quad f(S).
\end{equation}
The problem is to select a $k$-element subset of $V$ that maximizes $f$. This is a finite combinatorial optimization problem, so one way to solve it is by brute force: simply enumerate all possible subsets of size $k$, evaluate $f$, and pick the best subset. However, we are interested in cases arising from complex networks in which the number of possible subsets is very large. The number of possible subsets grows extremely fast as $V$ increases, so the brute force approach quickly becomes infeasible as $V$ becomes large. 

Other approaches for solving (\ref{optprob}) include general combinatorial optimization methods, e.g. branch and bound,  and methods based on $\ell_1$ optimization heuristics \cite{hassibi1999low}. However, general combinatorial methods do not scale well with problem size, and while $\ell_1$ heuristics scale well in principle, it is unclear when they work, and they provide no approximation guarantees. 

Another broad approach is to exploit structural properties of the set function $f$ that make them more amenable to optimization, such as modularity and submodularity \cite{lovasz1983submodular,fujishige2005submodular}. Submodularity is defined as follows.
\begin{definition}[Submodularity]
A set function $f: 2^V \rightarrow \mathbf{R}$ is called submodular if for all subsets $A \subseteq B \subseteq V$ and all elements $s \notin B$, it holds that
\begin{equation} \label{submod1}
f(A \cup \{s\}) - f(A) \geq f(B \cup \{s\}) - f(B),
\end{equation}
or equivalently, if for all subsets $A,B \subseteq V$, it holds that
\begin{equation} \label{submod2}
f(A) + f(B) \geq f(A\cup B) + f(A\cap B).
\end{equation}
\end{definition}
Intuitively, submodularity is a diminishing returns property where adding an element to a smaller set gives a larger gain than adding one to a larger set. A set function is called \emph{supermodular} if the reversed inequalities in (\ref{submod1}) and (\ref{submod2}) hold, and is called \emph{modular} if it is both submodular and supermodular, i.e. for all subsets $A,B \subseteq V$, we have $f(A \cap B) + f(A\cup B) = f(A) + f(B)$. A modular function has the following simple, equivalent characterization \cite{lovasz1983submodular}:
\begin{theorem}[Modularity] \label{modularity}
A set function $f: 2^V \rightarrow \mathbf{R}$ is modular if and only if for any subset $S \subseteq V$ it can be expressed as
\begin{equation} \label{submod3}
f(S) = w(\emptyset) + \sum_{s \in S} w(s)
\end{equation}
for some weight function $w : V \rightarrow \mathbf{R}$.
\end{theorem}
Modular set functions are analogous to linear functions and have the property that each element of a subset gives an independent contribution to the function value. In particular, modular functions are easily optimized: since the contribution of each element is independent, one can simply evaluate the set function for each individual element and then choose the best $k$ individual elements to obtain the best subset of size $k$. 
%The overall complexity is $n F$ where $F$ is the complexity of evaluating $f$ for a single element, plus sorting costs, which turn out to be negligible in our case.

Submodularity plays a similar role in combinatorial optimization as convexity in continuous optimization \cite{lovasz1983submodular}. It occurs often in applications (though is underexplored in systems and control theory); is preserved under various operations, allowing design flexibility; has a beautiful and practically useful mathematical theory; and there are efficient methods for minimizing and near-optimal methods for maximizing. 
%Submodular functions are NP-hard to maximize in general, but there are efficient algorithms for approximating the optimal solution to within a constant. 
%Connections between submodularity and matroid theory and independence in linear algebra (+ controllability?). 

We now demonstrate the modularity of a class of controllability metrics involving linear functions of the controllability Gramian.

\subsection{Main Result: Any linear function of the controllability Gramian is a modular set function}
Suppose we are given a stable system $A$ matrix and a set of possible $B$ matrix columns $V = \{b_1,..., b_M \} $. The problem is to choose a subset of the possible $B$ matrix columns to maximize a metric of controllability. Here, we consider the weighted trace of the controllability Gramian, i.e., any linear function of the controllability Gramian. For a given $S \subseteq V$, we form $B_S = [b_s], \ s \in S$ and the associated controllability Gramian $W_S  = \int_0^\infty e^{A \tau} B_S B_S^T e^{A^T \tau} d\tau$, which is the unique positive definite solution the Lyapunov equation
\begin{equation}
AW_S + W_S A^T + B_S B_S^T = 0.
\end{equation}
To simplify notation, we write $W_s$ for $W_{\{s\} }$. We have the following result.
\begin{theorem} \label{contrmod}
Let $A \in \mathbf{R}^{n \times n}$ be a stable dynamics matrix and $V = \{b_1,..., b_M \} $ be a set of possible actuator locations. The set function mapping subsets $S \subseteq V$ to any linear function of the associated controllability Gramian, i.e. $f(S) = \mathbf{tr}(\bar{C} W_S)$ for any weighting matrix $\bar{C}\in \mathbf{R}^{n \times n}$, is modular.
\end{theorem}
\begin{proof} 
We will prove the result directly using Theorem \ref{modularity}. Take any $S \subseteq V$ and let $B_S$ denote the input matrix formed by taking the associated columns defined by $S$. Note that 
\begin{equation}
\begin{aligned}
B_S B_S^T &= \left[\begin{array}{ccc}b_{s_1} & \cdots & b_{s_{|S|}}\end{array}\right] \left[\begin{array}{ccc}b_{s_1} & \cdots & b_{s_{|S|}}\end{array}\right]^T \\
		   &= \sum_{s \in S} b_s b_s^T.
\end{aligned}
\end{equation}
It is easy to see that the controllability Gramian associated with $B_S$ is simply a sum of the controllability Gramians associated with the individual columns of $B_S$:
\begin{equation}
\begin{aligned}
W_S  &= \int_0^\infty e^{A \tau} B_S B_S^T e^{A^T \tau} d\tau \\
           &= \int_0^\infty e^{A \tau} \sum_{s \in S} b_s b_s^T e^{A^T \tau} d\tau \\
           &= \sum_{s \in S} \int_0^\infty e^{A \tau} b_s b_s^T e^{A^T \tau} d\tau \\
           &= \sum_{s \in S} W_s 
\end{aligned}
\end{equation}
Now since trace is a linear matrix function, we have for any weight matrix $\bar{C} \in \mathbf{R}^{n \times n}$
\begin{equation}
\begin{aligned}
f(S) &= \textbf{tr}(\bar{C} W_S) \\
       &= \textbf{tr}\left(\sum_{s \in S} \bar{C} W_{s}\right) \\
       &= \sum_{s \in S} \textbf{tr} (\bar{C} W_{s})
\end{aligned}
\end{equation}
Thus, for any $s \in V$, we can define the weight function $w(s) = \textbf{tr}(\bar{C} W_s)$. Defining $w(\emptyset) = 0$, Theorem \ref{modularity} implies that $f(S) = \textbf{tr}(\bar{C} W_S)$ is a modular set function.
 %Finally, since $W_S \succeq 0$, then $trace(W_S) \geq 0$, which implies that $f$ is monotone. The proof is complete.
\end{proof}
Theorem \ref{contrmod} shows that each possible actuator placement gives an independent contribution to the trace of the controllability Gramian. Because of this, the actuator placement problem using this metric is easily solved: one needs only to compute the metric individually for each possible actuator placement, sort the results, and choose the best $k$. Based on the interpretations in the previous section, this means that placing actuators in a complex network to maximize the average amount of controllability available to move the system around the state space, or to maximize the energy in the system response to a unit impulse, is easily done. Since the result holds for the weighted trace, this gives considerable design freedom for actuator placement; important directions in the state space can be weighted and actuator placement done based on the weighted metric.

\subsection{A Dynamic Network Centrality Measure}
Network centrality measures are real-valued functions that assign a relative ``importance'' to each node within a graph. Examples include degree, betweenness, closeness, and eigenvector centrality. The meaning of importance and the relevance of various metrics depends highly on the modeling context. For example, PageRank, a variant of eigenvector centrality, turns out to be a much better indicator of importance than vertex degree in the context of networks of web pages, one of the core factors leading to Google's domination of web search.

In the context of complex dynamical networks, the controllability metric described above can be used to define a control energy-based centrality measure, describing the importance of a node in terms of its ability to move the system around the state space with a low-energy time-varying control input. In particular, imagine that it is possible to place an actuator at each individual node in the network; thus, define $V = \{e_1,...,e_n \}$, where $e_i$ is the standard unit basis vector in $\mathbf{R}^n$, i.e. $e_i$ has a $1$ in the $i$th entry and zeros elsewhere. We define the Control Energy Centrality for a complex dynamical network as follows.
\begin{definition}[Average Energy Controllability Centrality]
Given a complex network with $n$ nodes and an associated stable linear dynamics matrix $A \in \mathbf{R}^{n\times n}$, the Average Energy Controllability Centrality measure for node $i$ is given by
\begin{equation}
C_{CE}(i) = \mathbf{tr}(W_i), \quad i \in V
\end{equation}
where $W_i$ is the infinite-horizon controllability Gramian that satisfies $AW_i + W_i A^T + e_i e_i^T = 0$.
\end{definition}
An interesting topic for future work would be to explore the distribution of the Average Energy Controllability Centrality measure in random networks and networks from various application domains.

%%% Numerical experiments %%%
\section{Power electronic actuator placement in the European power grid}
% intro
New power electronic actuators, such as high voltage direct current (HVDC) links or flexible alternating current transmission devices (FACTS), can be used to improve transient stability properties in power grids by modulating active and reactive power injections to damp frequency oscillations and prevent rotor angle instability \cite{fuchs2011}. In this section, we illustrate the results via placement of such power electronic actuators in a model of the European power grid. We emphasize that this section is intended only to illustrate the theory in the preceding sections and show what kind of questions could be answered; many practical political and economic issues are neglected, and placements are evaluated entirely from a controllability perspective. 

% system model
We consider a simplified model of the European grid derived from \cite{Haase2006} with 74 buses, each of which is connected to a generator and a constant impedance load. We consider the placement of HVDC links, which are modeled as ideal current sources that can instantaneously inject AC currents into each bus; for modeling details see \cite{fuchs2011,fuchs2013a,fuchs2013b} The system dynamics we consider here are based on the swing equations, a widely-used nonlinear model for the time evolution of rotor angles and frequencies of each generator in the network \cite{kundur1994power}. Each HVDC link has three degrees of freedom that allow influence of the frequency dynamics at the corresponding buses. The nonlinear model is linearized\footnote{Ideally, one would of course want to evaluate actuator placement on the nonlinear model, but even evaluating controllability metrics can be extremely difficult computationally, even for small-scale nonlinear systems. This section is intended to illustrate the theory from the previous section, so we focus on a linearized model, though actuator placement problems for nonlinear networks are an important topic for future work.} about a desired operating condition for each possible HVDC link placement, and the placements are evaluated based on the linearized model.

% placement results and commentary
Each generator has two associated states: rotor angle and frequency, which gives a 148-dimensional state space model, i.e., $A \in \mathbf{R}^{148\times 148}$. Since an HVDC link could be placed in principle between any two distinct nodes in the network, there are 2701 possible locations. To get an idea about the size of the search space, consider the problem of finding the best subset of size 10. This gives approximately $5.6\times 10^{27}$ possible combinations, far too many for a brute force search. As we have seen, the modularity property allows us to consider each placement individually.

Figure 1 shows the network and the 10 best placements according to the controllability Gramian trace metric with all state space directions weighted equally, i.e., $C = I_{148}$. The best two are relatively long lines connecting the northeast-southwest and northwest-southeast quadrants of the network, respectively. Interestingly, the next group of placements is concentrated in the southeast, indicating that there is room to improve control authority by increasing connectivity in this sparsely connected region. This also indicates a potential weakness in the trace metric, which may cluster actuators to get high controllability in a few directions at the expense of controllability in other directions. Figure 2 shows the sorted distribution of the metric, with the top few placements giving a substantial benefit over other placements. Figure 3 shows the 10 best placements according to the controllability Gramian trace metric, but with the frequency dynamics in the network weighted equally and the rotor angle dynamics ignored, i.e $C = I_{74} \otimes [0, \ 1]$. In this case, the optimal placements are more evenly distributed in the network.
%
%\begin{figure}
%\begin{center} 
%\resizebox{0.99\linewidth}{!}{\includegraphics{./Figures/Europe_Grid_Map.pdf}}
%\caption{Nodes of European power grid model.}
%\end{center}
%\end{figure}

\begin{figure}
\begin{center} 
\resizebox{0.99\linewidth}{!}{\includegraphics{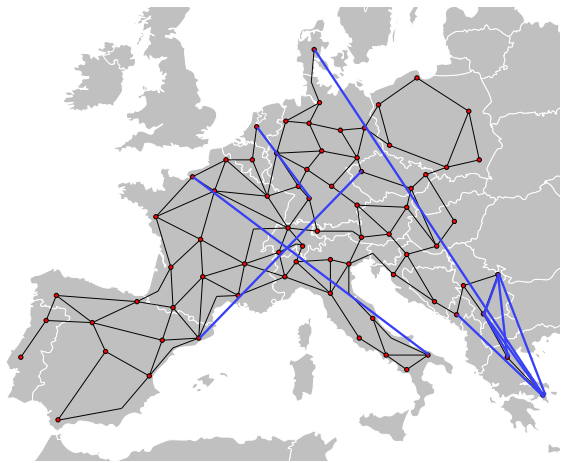}}
\caption{Best 10 HVDC line placements (in blue) according to the controllability Gramian trace metric.}
\end{center}
\end{figure}

\begin{figure}
\begin{center} 
\resizebox{0.89\linewidth}{!}{\includegraphics{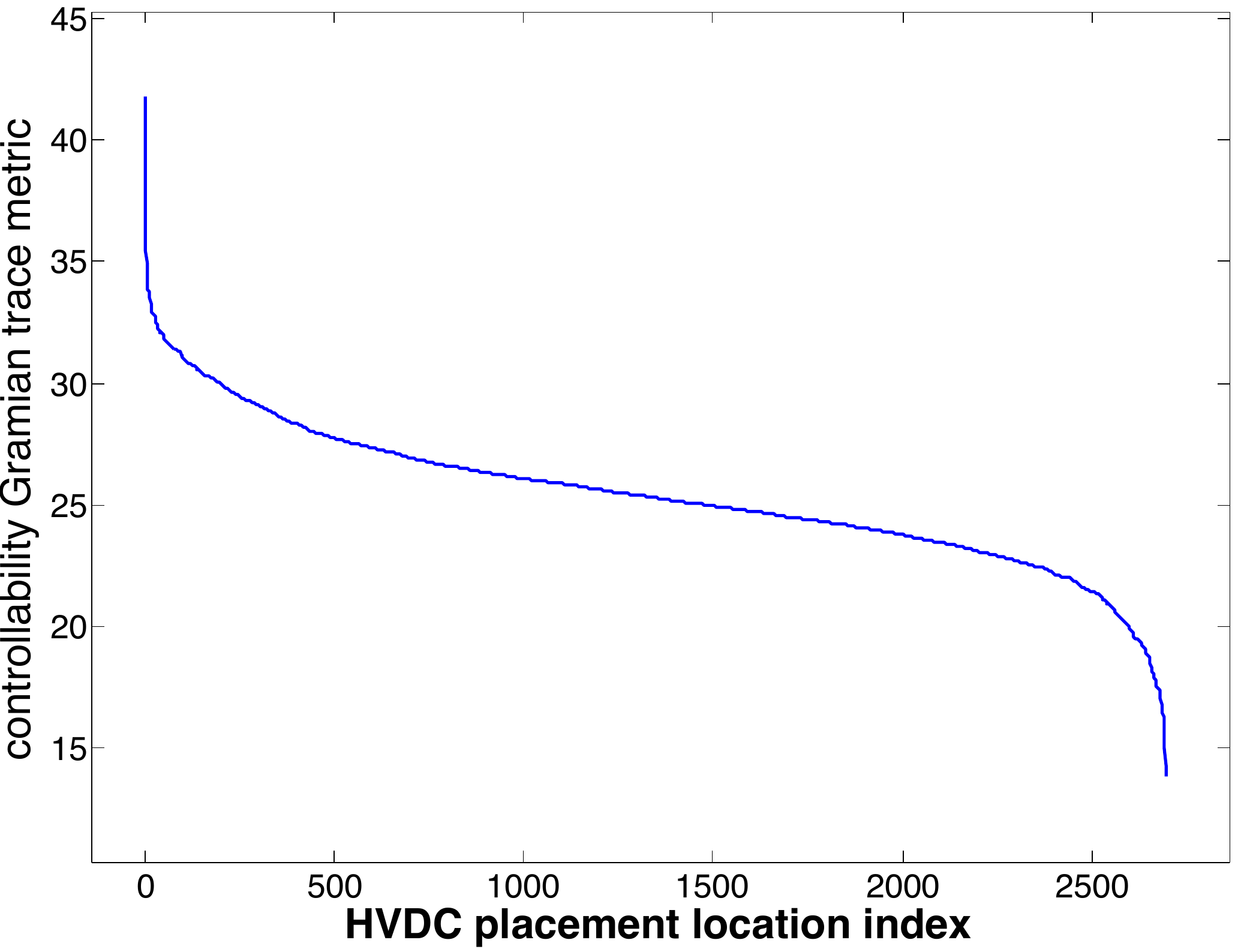}}
\caption{Distribution of the controllability Gramian trace metric.}
\end{center}
\end{figure}

\begin{figure}
\begin{center} 
\resizebox{0.99\linewidth}{!}{\includegraphics{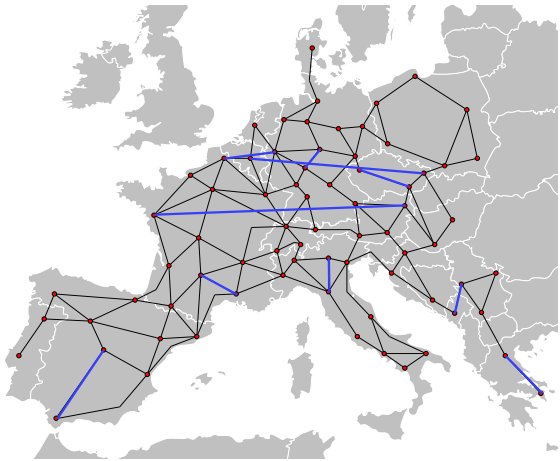}}
\caption{Best 10 HVDC line placements according to the controllability Gramian trace metric, weighted for frequency dynamics.}
\end{center}
\end{figure}

%\subsection{Promotion targets in a Genetic Regulatory Network}
%Maybe show corresponds well with Sean's results using more complicated nonlinear reachability analysis.
%
%
%\subsection{Optimal Advertising in a Social Network}
%Get model from Andreas Krause or come up with some kind of linear diffusion model.

%%% Conclusions and Outlook %%%
\section{Conclusions and Outlook}
We have considered optimal actuator and sensor placement problems in complex dynamical networks. These problems are in general difficult combinatorial optimization problems; however, we have shown that an important class of metrics related to the controllability and observability Gramians yield modular set functions and are therefore efficiently globally optimized. We also defined the Average Energy Controllability Centrality measure, which assigns an importance value to each node in a dynamical network based on its ability to move the system around the state space with a low-energy time-varying control input. The results were illustrated via placement of power electronic actuators in a model of the European power grid. 

There are many open problems involving the structure of combinatorial optimization problems in the optimal placement of sensors and actuators in complex networks. What other linear system controllability or observability metrics have exploitable combinatorial structure, e.g., modularity or submodularity? Further results have been obtained in \cite{summers2014}. Our ongoing work is exploring other case studies in power networks, biological networks, social networks, and discretized models of infinite-dimensional systems. Finally, more complicated systems (e.g. constrained, nonlinear, hybrid, etc.) require a full reachability analysis in general, which does not scale well, but one could explore how efficient methods could be used to obtain approximate metrics in these types of systems.

\section*{Acknowledgements}
The authors would like to thank Dr. Alexander Fuchs for providing details and helpful discussion about the power system model discussed in Section IV.

\bibliographystyle{plain}  
\bibliography{refs}  

\end{document}